\documentclass[11pt,a4paper,twoside]{amsart}

\usepackage[english]{babel}

\usepackage{amssymb}
\usepackage{amsmath}
\usepackage{amsthm}
\usepackage[dvips]{graphicx}
\usepackage[textsize=tiny]{todonotes} 

\theoremstyle{plain}
\newtheorem{thm}{Theorem}
\newtheorem{prop}[thm]{Proposition}
\newtheorem{cor}[thm]{Corollary}
\newtheorem{lemma}[thm]{Lemma}
\newtheorem{defin}[thm]{Definition}

\newtheorem{rmk}[thm]{Remark}

\renewcommand{\epsilon}{\varepsilon}
\DeclareMathOperator{\supp}{supp}
\newcommand{\dist}{\mathrm{dist}}

\newcommand{\abs}[2][{}]{\lvert{#2}\rvert_{#1}}    

\newcommand{\normsymb}{\|}

\newcommand{\norm}[2]{\normsymb{#1}\normsymb_{#2}}  

\newcommand{\RR}{\mathbb{R}} 
\newcommand{\CC}{\mathbb{C}} 
\newcommand{\NN}{\mathbb{N}} 
\newcommand{\ZZ}{\mathbb{Z}} 
\newcommand{\TT}{\mathbb{T}} 

\DeclareMathOperator{\dd}{d\!}  

\title[]{On null-controllability of the heat equation on infinite strips and control cost estimate}

\author[]{Michela Egidi}
\address{Technische Universit\"at Dortmund, Fakult\"at f\"ur Mathematik, 44221 Dortmund, Germany\\
}
\email{michela.egidi@mathematik.tu-dortmund.de}

\begin{document}

\maketitle

\begin{abstract}
We consider an infinite strip $\Omega_L=(0,2\pi L)^{d-1}\times\RR$, $d\geq 2$, $L>0$, 
and study the control problem of the heat equation on $\Omega_L$ with Dirichlet or Neumann boundary conditions, 
and control set $\omega\subset\Omega_L$. 
We provide a sufficient and necessary condition for null-controllability in any positive time $T>0$, 
which is a geometric condition on the control set $\omega$. 
This is referred to as ``thickness with respect to $\Omega_L$'' and implies that the set $\omega$ 
cannot be concentrated in a particular region of $\Omega_L$. 
We compare the thickness condition with a previously known necessity condition for null-controllability and 
give a control cost estimate which only shows dependence on the geometric parameters of $\omega$ and the time $T$.  
\end{abstract}

\section{Introduction and main results}\label{intro}

Let $L>0$ and let $T>0$ be a fixed positive time. Let $\TT^{d-1}_L:=(0,2\pi L)^{d-1}$ be an open $(d-1)$-dimensional cube with sides of length $2\pi L$ 
and let $\Omega_L:=\TT^{d-1}_L\times\RR$ be an infinite strip in $\RR^d$. 

For any given subset $\omega\subset\Omega_L$, we consider the controlled heat equation on $\Omega_L$ with \emph{control set} $\omega$, i.e. the system 

\begin{equation}\label{eq:heat_equation}
\left\{\begin{array}{ll}
\partial_t u(t,x) -\Delta u(t,x) = \chi_{\omega}(x)v(t,x) & \qquad\text{ on } (0,T)\times\Omega_L\\
u(0,\cdot)=u_0(\cdot)\in L^2(\Omega_L) & \qquad\text{ in } \Omega_L,
\end{array}
\right.
\end{equation}
where $\Delta$ denotes the Laplacian on $\Omega_L$ with either Dirichlet or Neumann boundary conditions, and $\chi_{\omega}$ is the characteristic function of $\omega$. 
The function $v\in L^2((0,T)\times \Omega_L)$ is called \emph{control function}.

System \eqref{eq:heat_equation} is said to be \emph{null-controllable} in time $T>0$ if for every initial data $u_0\in L^2(\Omega_L)$ 
there exists a control function $v\in L^2((0,T)\times\Omega_L)$ such that the solution of \eqref{eq:heat_equation} satisfies $u(T,\cdot)=0$.

In case null-controllability holds in time $T>0$, for all initial data $u_0$ the set 
\[
\mathcal{C}_{u_0,T}=\{v\in L^2((0,T)\times\Omega_L)\;\vert\; \text{ the solution } u \text{ of \eqref{eq:heat_equation} satisfies } u(T,\cdot)=0\} 
\]
is not empty and the quantity
\begin{equation}\label{eq:control_cost}
C_T:=\sup_{\norm{u_0}{L^2(\Omega_L)}=1}\inf_{v\in \mathcal{C}_{u_0,T}}\norm{v}{L^2((0,T)\times \omega)}
\end{equation}
is called \emph{control cost}.

It is well known that the heat equation on bounded domains $\Omega$ with open control set $\omega\subset \Omega$ is null-controllable in any time $T>0$, 
see for example \cite{lebeau-robbiano-95}. 
It has also been recently shown in \cite{ApraizEWZ-14} and \cite{EscauriazaMZ-15} that if $\Omega$ is bounded and $\omega$ is a measurable subset of non-zero measure, 
null-controllability still holds.

For unbounded domains the situation is different. For the heat equation on $\RR^d$, $d\geq 1$, a sharp necessary and sufficient condition for null-controllability has been
recently established in \cite{EV-17} and \cite{WWZZ:17} independently. This condition is referred to as $(\gamma, a)$-\emph{thickness} and means that the set is somehow well-distributed in $\RR^d$ (see Definition \ref{def:thickness}~(i) below).

%
%

More generally, for arbitrary unbounded Euclidean domains $\Omega$ a necessary condition for null-controllability connected to the heat kernel of the
Laplacian on $\Omega$ with Dirichlet boundary conditions
has been identified in \cite[Theorem 1.11]{miller:05}. 
Precisely, let us consider system \eqref{eq:heat_equation} on a given unbounded Euclidean domain $\Omega$, instead of on $\Omega_L$, 
with control set $\omega\subset\Omega$. 
If there exist a sequence of points $(y_n)_{n\in\NN}$ in $\Omega$, a time $\bar{T}>0$ and a constant $\kappa>1$ such that 
\begin{equation}\label{eq:necessary-miller}
-2T\log\left(\int_\omega \exp\left( -\frac{\norm{x-y_n}{2}^2}{2T}\right) \dd x \right) 
      -\kappa \frac{\pi^2 d^2}{4}\left( \frac{T}{d_b(y_n,\partial\Omega)} \right)^2 \underset{n\to\infty}{\longrightarrow} + \infty,
\end{equation}
where $d_b(y_n,\partial\Omega)=\min \left( \dist(y_n, \partial\Omega), \frac{T\pi^2 d}{4}\right)$, 
then the controlled heat equation on $\Omega$ is not null-controllable in any time $T< \bar{T}$. 
Here $\dist$ denoted the distance function on $\Omega$.

In particular, the author establishes the failure of null-controllability if the control set $\omega$ has finite Lebesgue measure. 

Motivated by the recent work \cite{EV-17} and \cite{WWZZ:17}, we show that a local notion of $(\gamma, a)$-thickness is a sufficient 
and necessary condition for null-controllability of system \eqref{eq:heat_equation}.

Let us now introduce the geometric definitions needed to state our main theorem. 
In what follow, $\abs{\cdot}$ denotes the Lebesgue measure. 

\begin{defin}\label{def:thickness}
\begin{itemize}
\item[(i)] A measurable set $S\subset\RR^d$ with positive measure is called \emph{thick} if there exist 
$\gamma\in(0,1]$ and $a\in (\RR_+)^d$ such that for all $P\subset\RR^d$ hyperrectangles with sides 
parallel to coordinate axes and of length $a_1,\ldots, a_d$ we have
    \begin{equation}\label{eq:thick}
    \abs{S\cap P} \geq \gamma \abs{P}. 
    \end{equation}
\item[(ii)] A measurable set $S\subset\Omega_L$ with positive measure is called \emph{thick with respect to} $\Omega_L$ 
if there exist $\gamma\in(0,1]$ and $a\in (\RR_+)^d$ such that for all $P\subset\Omega_L$ hyperrectangles with sides 
parallel to coordinate axes and of length $a_1,\ldots, a_d$ we have
    \begin{equation}\label{eq:thick-wrt}
    \abs{S\cap P} \geq \gamma \abs{P}.
    \end{equation}
   \end{itemize}
 To emphasise the parameter we also refer to $S$ as $(\gamma, a)$-thick or $(\gamma, a)$-thick with respect to $\Omega_L$. 
\end{defin}

We point out that $(\gamma,a)$-thickness with respect to $\Omega_L$ implies that $a_j\leq 2\pi L$ for all $j=1,\ldots d-1$. 

Examples of thick sets with respect to $\Omega_L$ are periodic arrangements of balls inside the strip, 
sets of type $M\times \RR$ with section $M\subset \TT^{d-1}_L$ being a non-empty measurable subset of positive measure,
and $(\gamma, a)$-thick sets $S\subset \RR^d$ with $a_j\leq 2\pi L$ for $j=1,\ldots, d$. 

We observe that it is always possible to obtain a thick set starting from a set thick with respect to $\Omega_L$.

\begin{lemma}\label{lemma:thickness}
Let $S\subset\Omega_L$ be a $(\gamma, a)$-thick set with respect to $\Omega_L$, then the set $\tilde{S}=S\cup (\RR^d\setminus \Omega_L)$
is $(\gamma/2^d, 2a)$-thick in $\RR^d$. 
\end{lemma}

\begin{proof}
Let $S,\tilde S$ be as in the statement of the lemma.
Let $P$ be a hyperrectangle in $\RR^d$ with sides parallel to coordinate axes and of length $2a_1, \ldots, 2a_d$. 
Then $P$ always contains a hyperrectangle $Q$ with sides parallel to coordinate axes and of length $a_1,\ldots, a_d$.
Moreover, we have $\abs{Q}=1/2^d\abs{P}$. 
 
Case 1: $P\subset\RR^d\setminus\Omega_L$. 
Then $\abs{P\cap \tilde{S}}=\abs{P}\geq (\gamma/ 2^d)\abs{P}$ since $\gamma/ 2^{d}\leq 1$. 

Case 2: $P\subset \Omega_L$. 
Let $Q\subset P$ as above, then $\abs{P\cap \tilde{S}}\geq \abs{Q\cap S}\geq \gamma\abs{Q}=(\gamma/ 2^{d})\abs{P}$, 
since $S$ is $(\gamma, a)$-thick with respect to $\Omega_L$.

Case 3: $P\cap \partial\Omega_L\neq \emptyset$.  
Then, the hyperrectangle $Q$ is either contained in $\Omega_L$ or in $\RR^d\setminus\Omega_L$. 
If $Q\subset \Omega_L$, we proceed as in Case 1. 
If $Q\subset (\RR^d\setminus\Omega_L)$, we proceed as in Case 2. 
In both cases we then obtained $\abs{P\cap \tilde{S}}\geq (\gamma/ 2^d)\abs{P}$. 

Therefore the set $\tilde S$ is $(\gamma/2^d, 2a)$-thick, as claimed.
\end{proof}

\begin{thm}\label{thm:main}
Let $T>0$ and consider system \eqref{eq:heat_equation} with control set $\omega\subset\Omega_L$. 
The following statements are equivalent: 
\begin{itemize}
\item [(i)] $\omega$ is thick with respect to $\Omega_L$,
\item [(ii)] system \eqref{eq:heat_equation} is null-controllable in any time $T>0$. 
\end{itemize}
Moreover, if $\omega$ is a $(\gamma,a)$-thick set with respect to $\Omega_L$, the control cost satisfies
    \begin{equation} \label{eq:control-cost}
    C_T\leq \left(\frac{(4K)^d}{\gamma}\right)^{12\sqrt{2}K(2\norm{a}{1}+d)} \exp\left(\frac{(48 K)^2(2\norm{a}{1}+d)^2\log^2((4K)^d/\gamma)}{2T}\right),
    \end{equation}
    where $\norm{a}{1}= \sum_{j=1}^d a_j$, and $K>0$ is a universal constant.
\end{thm}

The proof of the necessity condition in Theorem \ref{thm:main} is inspired by \cite{miller:05} and build upon heat kernel estimates. 
Indeed, condition \eqref{eq:necessary-miller} is equivalent to $\omega$ not being thick with respect to $\Omega_L$,  
as Lemma \ref{lemma:equivalence} in Section \ref{sec:necessity} shows. 

The sufficiency condition of Theorem \ref{thm:main} and the control cost estimate are a consequence of the
following proposition and Lemma \ref{lemma:thickness}.

\begin{prop}\label{prop:sufficiency}
    Let $S\subset\RR^d$ be a $(\gamma, a)$-thick set with $a_j\leq 2\pi L$ for $j\in\{1,\ldots,d-1\}$, 
    and let $\omega=S\cap \Omega_L\subset \Omega_L$. 
    Then, system \eqref{eq:heat_equation} is null-controllable in any time $T>0$, and the control cost satisfies
    \begin{equation} \label{eq:control-cost}
    C_T\leq \left(\frac{(2K)^d}{\gamma}\right)^{12\sqrt{2}K(\norm{a}{1}+d)} \exp\left(\frac{(48 K)^2(\norm{a}{1}+d)^2\log^2((2K)^d/\gamma)}{2T}\right),
    \end{equation}
    where $\norm{a}{1}= \sum_{j=1}^d a_j$, and $K>0$ is a universal constant.
\end{prop}

Indeed, to prove the implication $(i)\Rightarrow (ii)$ of Theorem \ref{thm:main}, we consider the set $\omega$
of the statement and we construct $\tilde{\omega}=\omega\cup (\RR^d\setminus\omega)$, which by Lemma \ref{lemma:thickness}
is thick in $\RR^d$. 
Then, the above proposition applied with $\tilde{\omega}$ instead of $\tilde S$ gives us null-controllability 
of the system. The control cost estimate is then obtained pluggin into \eqref{eq:control-cost} the thickness parameters 
of $\tilde \omega$ given by Lemma \ref{lemma:thickness}.

\smallskip 

If the set $S$ respects the cartesian structure of $\RR^d=\RR^{d-1}\times\RR$, we obtain
the following corollary.

\begin{cor}\label{cor:1}
\begin{itemize}
 \item[(i)] Let $S\subset \RR^{d-1}$ be a $(\gamma, a)$-thick set with $a_j\leq 2\pi L$ for $j\in\{1,\ldots,d-1\}$. 
    Then, the controlled heat equation \eqref{eq:heat_equation} on $\Omega_L$ with control set $\omega= (S\times\RR)\cap\Omega_L$ 
    is null-controllable in any time $T>0$ with control cost
    \begin{equation*}
    C_T\leq \left(\frac{(2K)^d}{\gamma}\right)^{12\sqrt{2}K(\norm{\tilde{a}}{1}+d)} \exp\left(\frac{(48 K)^2(\norm{\tilde{a}}{1}+d)^2\log^2((2K)^d/\gamma)}{2T}\right),
    \end{equation*} 
    where $\tilde{a}=(a_1,\ldots, a_{d-1},\alpha)$ for any finite number $\alpha>0$ and $K>0$ is the universal constant from Proposition \ref{prop:sufficiency}.\\
    
\item[(ii)] Let $S_1\subset\RR^{d-1}$ be a $(\gamma_1, a_1)$-thick set with $a_j\leq 2\pi L$ for $j\in\{1,\ldots,d-1\}$ 
and let $S_2\subset \RR$ be a $(\gamma_2,a_2)$-thick set.
    Then, the controlled heat equation \eqref{eq:heat_equation} on $\Omega_L$ with control set $\omega= (S_1\times S_2)\cap\Omega_L$ 
    is null-controllable in any time $T>0$ with control cost
    \begin{equation*}
    C_T\leq \left(\frac{(2K)^d}{\gamma_3}\right)^{12\sqrt{2}K(\norm{a_3}{1}+d)} \exp\left(\frac{(48 K)^2(\norm{a_3}{1}+d)^2\log^2((2K)^d/\gamma_3)}{2T}\right),
    \end{equation*} 
    where $\gamma_3=\gamma_1\gamma_2$, $a_3=(a_1,a_2)\in\RR^d$, and $K>0$ is the universal constant from Proposition \ref{prop:sufficiency}.
\end{itemize}
\end{cor}

The two statement are a straighforward consequence of the facts that $S\times\RR$ is $(\gamma, \tilde a)$-thick, 
and $S_1\times S_2$ is $(\gamma_1\gamma_2, (a_1,a_2))$-thick.

\begin{rmk}\label{rmk:generalization}
Let $n, m\geq1$, $R=(R_1,\ldots, R_m)\in(\RR_+)^m$, and consider the cartesian products $M\times\RR^n$ or $\RR^n\times M$, for
$M=(0,2\pi R_1)\times\ldots\times(0, 2\pi R_{m})$. 
Then, Proposition \ref{prop:sufficiency} and Theorem \ref{thm:main}, as well as Corollary \ref{cor:1}, 
have an analogous formulation for the heat equation on 
$M\times\RR^n$ or $\RR^n\times M$ controlled from an interior region $\omega$. 
\end{rmk}

Let us now briefly comment on the estimate of the control cost. 
Geometric bounds on the control cost have been previously obtained for small times in \cite{miller:04}, 
see also \cite{FCZ:00,zuazua:01}, for the heat equation on $d$-dimensional, compact, connected manifolds 
controlled from an open interior region $\omega\subset M$. 
They showed
\[
 0 < \sup_{y\in M} \dist(y, \bar{\omega})^2/4 \leq \liminf_{T\to 0} T\log C_T \leq \limsup_{T\to 0} T\log C_T < +\infty ,
\]
where $\dist$ denotes the distance function on $M$. 
In our case, i.e. for $\Omega_L$ and $\omega$ $(\gamma, a)$-thick with respect to $\Omega_L$, 
we do not achieve a lower bound, but an upper bound, namely,
\[
 \limsup_{T\to 0} T\log C_T \leq \frac{(48K)^2}{2}(2\norm{a}{1}+d)^2\log^2\left(\frac{(4K)^d}{\gamma}\right).
\]

Moreover, bounds on the control cost for a more general class of controlled systems have been investigate in \cite{miller:10}, 
where the author focuses again on small times. His result (see \cite[Theorem 2.2]{miller:10}) applies to our setting and 
gives the control cost 
\[
C_T\leq 4\left(\frac{(4K)^d}{\gamma}\right)^{6d-1}\exp\Big(\frac{c}{T}\Big)
\] 
valid for all times $T\in(0,T_0)$, where $c$ and $T_0$ implicitly depend on the model parameters, 
and $K>0$ is the universal constant from Proposition \ref{prop:sufficiency}.
Comparing this estimate with the bound in \eqref{eq:control-cost}, we see that we have gained an estimate for small and large times, 
which is explicit in the dependence on $\omega$ and $T$.  
This is due to the use of the observability result in \cite{B-PS:17}, see also Section \ref{sec:observability}.  

We also remark that the estimate in \eqref{eq:control-cost}, as well as the bound on the limit superior, is independent of the scale $L$.

\subsection*{Open questions and further investigations}
The sufficiency results here presented open questions in two directions. 
Since the infinite strip is an unbounded flat domain, it is natural to ask about other unbounded domains. 
For Euclidean domains which can be exhausted by rectangles, i.e. half-spaces and orthants, 
and for cones with angle $\pi/2^n$, $n\geq2$, null-controllability of the heat equation with control region 
given by the intersection of a thick set and the domain, 
and corresponding control cost estimates have been established in \cite[Section 2]{egidi-seelmann-19}.
More generally, the authors show that if the heat equation on a domain in $\RR^d$ symmetric with respect to a hyperplane 
is null-controllable, then so is the heat equation on the symmetric parts. 
Different is the situation for unbounded domains in $\RR^d$ with curved boundary, for example paraboloids. 
Indeed, in this case, null-controllability of the controlled heat equation is an open problem for any kind 
of control set. To the best of our knowledge, not even a spectral inequality for functions in the range 
of the spectral projector of the associated Laplace operator is known. 

The other natural question is on the type of the operator considered.
Namely, one can study system \eqref{eq:heat_equation} with a second order elliptic operator of the form 
$L u= \sum_{i,j=1}^{d}\partial_i a_{ij}\partial_j u + \sum_{j=1}^{d} b_j \partial_j u +c u$, 
where $a_{ij},b_j,c$ are measurable functions, instead of the Laplacian. 

The null-controllability of a heat-like system on the whole space with operator $L$ having constant and uniformly elliptic 
coefficients $a_{i,j}$ and $b_j=c=0$ for all $j=1,\ldots, d$ has been recently proved in \cite{gallaun-seifert-taut-19}, 
through an abstract observability result similar to the ones in \cite{TT:11, B-PS:17}, but allowing for 
bounded operators instead of projections and assuming a spectral inequality and a dissipation estimate 
valid for large enough energy value instead of for all energy values.
This last feature is of particular use when dealing with an operator $L$ having constant and uniformly elliptic 
coefficients $a_{i,j}$ and non-zero constant coefficients $b_j$ and $c$, as it allows to easily modify the 
proof of \cite[Theorem 3.1]{gallaun-seifert-taut-19} to include this case \cite{gst-private}.

Under a deeper study of the properties of the operator $f\mapsto h_f$ (see beginning of Section \ref{sec:spectal-inequality}),
it is possible to combine the spectral inequality of Theorem \ref{thm:log-ser-strip}
with techniques developed in \cite{gallaun-seifert-taut-19} (see Theorem 2.1, the proof Theorem 3.3, and Corollary 4.6)
to conclude null-controllability of the heat-like system with elliptic operator $L$ in any time $T>0$ 
and obtain a control cost estimate explicitely depending on the model parameters.



Another question emerging from this work is whether the notion "thickness with respect to" is a necessary condition 
for controlled heat-like systems on a general unbounded Euclidean domain. 
Here we can give a positive answer as long as the corresponding operator
admits the existence of a heat kernel enjoying lower and upper gaussian bounds. 
For such existence criteria we refer the reader, for example, to \cite{davies:89, arendt-terElst-97, ouhabaz-04} 
and the references therein.
In this case, an adaptation of the proof in Section \ref{sec:necessity} gives the claim, under an appropriate 
modification of Definition \ref{def:thickness} (ii). 
Moreover, it is also possible to modify the proof of Lemma \ref{lemma:equivalence} accordingly, to show that 
condition \ref{eq:necessary-miller} is equivalent to "not being thick with respect to".

\subsection*{Organization of the paper}
The rest of the paper is organised as follow. In Section \ref{sec:observability} we discuss an observability result by K. Beauchard and K. Pravda-Starov 
\cite{B-PS:17} on which the proof of Proposition \ref{prop:sufficiency} is based. 
In Section \ref{sec:spectal-inequality} we derive a spectral inequality for a sub-class of $L^2(\Omega_L)$-functions. 
In Section \ref{sec:sufficiency} we prove Proposition \ref{prop:sufficiency}.
Finally in Section \ref{sec:necessity} we prove the implication $(ii)\Rightarrow(i)$ of
Theorem \ref{thm:main}, 
and we compare the notion of thickness with respect to $\Omega_L$ to the necessary condition \eqref{eq:necessary-miller}.

\subsection*{Acknowledgment} This work has been partially supported by the DFG Grant
Ve 253/7-1 ``Multiscale Version of the Logvinenko-Sereda Theorem''. The author would like to thank Ivan Veseli\'c for suggesting the topic and Albrecht Seelmann for comments
on a first draft and stimulating discussions. 

\section{An abstract observability result}\label{sec:observability}

The Hilbert Uniqueness Method, see for example \cite[Theorem 2.44]{coron:07}, establishes that 
null-controllability of system \eqref{eq:heat_equation} in time $T>0$ is equivalent to the following observability estimate with respect to $\omega$: 
\begin{equation}\label{eq:observability-def}
  \exists C>0:\;\forall\; g_0\in L^2(\Omega_L),\quad \norm{g(T,\cdot)}{L^2(\Omega_L)}^2\leq C\int_0^T \norm{g(t,\cdot)}{L^2(\omega)}^2\dd t,
\end{equation}
where $g$ is the solution of the adjoint system 
\begin{equation}\label{eq:heat_equation-adjoint}
  \left\{\begin{array}{ll}
  \partial_t g(t,x) - \Delta g(t,x) = 0 & \qquad\text{ on } (0,T)\times\Omega_L\\
  g(0,\cdot)=g_0(\cdot)\in L^2(\Omega_L) & \qquad\text{ in }\Omega_L.
  \end{array}\right.
\end{equation}
In addition, it provides an estimate for the control cost. In fact, 
\[
C_T\leq \sqrt{C}, 
\]
where $C$ is the observability constant in \eqref{eq:observability-def}.

Therefore, to show null-controllability, we use an abstract observability result obtained in \cite[Theorem 2.1]{B-PS:17}, based on the 
Lebeau-Robbiano strategy, see \cite{lebeau-robbiano-95}, which for self-adjoint operators is just the Lebeau-Robbiano strategy
adapted to unbounded domains. 
Such observability result can also be found in \cite{miller:10} (Theorem 2.2), which however holds only for small time intervals.  
The observability result below holds under the assumption of a dissipation inequality and a spectral inequality for orthogonal 
projections on $L^2(\Omega)$, which are not necessarily related to the Laplace operator. 
If these orthogonal projections are chosen as the spectral projectors of the operator under 
consideration, the dissipation inequality holds trivially (by functional calculus) and a first concrete instance 
of such a spectral inequality involving open sets with an additional property is given in \cite{lerousseau-mojano-16}.

%
\begin{thm}\label{thm:observability-B-PS}
      Let $\Omega$ be an open subset of $\RR^d$, $\omega$ be a measurable subset of $\Omega$ with positive measure, 
      $(\pi_k)_{k\in\NN}$ be a family of orthogonal projections on $L^2(\Omega)$, 
      $(e^{t\Delta})_{t\geq 0}$ be the contraction semigroup associated to the Laplacian on $L^2(\Omega)$, and let $c_1,c_2,\eta_1,\eta_2,t_0,m>0$  
      be positive constants with $\eta_1 < \eta_2$. If the spectral inequality
      \begin{equation}\label{eq:spectral-inequality-general}
      \forall \ g\in L^2(\Omega), \ \ \forall \ k\in\NN,\qquad \norm{\pi_k g}{L^2(\Omega)}\leq e^{c_1 k^{\eta_1}}\norm{\pi_k g}{L^2(\omega)},
      \end{equation}
      and the dissipation estimate
      \begin{equation}\label{eq:dissipation-general}
      \forall \ g\in L^2(\Omega), \ \forall \ k\in\NN, \forall \ 0< t<t_0, \quad \norm{(1-\pi_k)(e^{t\Delta}g)}{L^2(\Omega)}\leq \frac{e^{-c_2 t^m k^{\eta_2}}}{c_2}\norm{g}{L^2(\Omega)}
      \end{equation}
      hold, then there exist two positive constants $C_1, C_2>0$ such that the following observability estimate holds
      \begin{equation}\label{eq:observability-general}
      \forall \ T>0, \ \forall\  g\in L^2(\Omega), \quad \norm{e^{T\Delta}g}{L^2(\Omega)}^2\leq C_1\exp\left(\frac{C_2}{T^{\frac{\eta_1 m}{\eta_2-\eta_1}}}\right)\int_0^T \norm{e^{t\Delta}g}{L^2(\omega)}^2 \dd t.
      \end{equation}
\end{thm}

We point out that the original statement is formulated for an open set $\omega$ and was presented with a unified constant $C=\sup(C_1, C_2)$. 
However the statement is still valid when the assumption on $\omega$ is relaxed to measurability. Distinguishing the two constants 
allows for a more precise behaviour of the control cost in terms of the geometric parameters.

\section{Spectral inequality on infinite strips}\label{sec:spectal-inequality}

Let $R=(R_1, \ldots, R_{d-1})\in (\RR_+)^{d-1}$ and consider $\Omega_R=\TT^{d-1}_R\times\RR$, 
where $\TT^{d-1}_R:=(0,2\pi R_1)\times\ldots\times(0,2\pi R_{d-1})$. With abuse of notation, we write  
$(\frac{1}{R}\ZZ)^{d-1}:=\frac{1}{R_1}\ZZ\times\ldots\times\frac{1}{R_{d-1}}\ZZ$ and 
$\frac{k}{R}:=(\frac{k_1}{R_1},\ldots, \frac{k_{d-1}}{R_{d-1}})$ for $k\in\ZZ^{d-1}$. 

By Fourier Analysis, any function $f\in L^2(\Omega_R)$ can be represented as 
\[
f(x_1,x_2)=\sum_{\frac{k}{R}\in\left(\frac{1}{R}\ZZ\right)^{d-1}}
\left(\int_{\RR}  h_f\left(\frac{k}{R},\xi\right) e^{i x_2 \xi}\dd\xi\right)e^{i\frac{k}{R}\cdot x_1},\quad (x_1,x_2)\in \TT^{d-1}_R\times\RR
\]
where 
\[
h_f:\left(\frac{1}{R}\ZZ\right)^{d-1}\times\RR \rightarrow \CC,\quad
h_f\left(\frac{k}{R},\xi\right) = \frac{1}{\sqrt{2\pi}\abs{\TT^{d-1}_R}}\int_{\TT^{d-1}_R}\int_\RR f(s,t) e^{-i \xi t} e^{-i \frac{k}{R}\cdot s} \dd t \dd s,
\]
and $\frac{k}{R}\cdot s$ stands for the Euclidean inner product in $\RR^{d-1}$.

In this section we consider functions $f\in L^2(\Omega_R)$ with $\supp h_f\subset J_1\times J_2$, 
where $J_1\subset\RR^{d-1}$ is a $(d-1)$-dimensional hyperrectangle with sides of length $b_1,\ldots, b_{d-1}$ 
and parallel to coordinate axes, and $J_2\subset\RR$ is an interval of length $b_d$. 
We assume both $J_1$ and $J_2$ to be centred at zero.
These functions have then the following representation
\begin{equation}\label{eq:function}
f(x_1,x_2)=\sum_{\frac{k}{R}\in\left(\frac{1}{R}\ZZ\right)^{d-1}\cap J_1}
\left(\int_{J_2}  h_f\left(\frac{k}{R},\xi\right) e^{i x_2 \xi}\dd\xi\right)e^{i\frac{k}{R}\cdot x_1}.
\end{equation}
Since the Fourier frequencies of $f(x_1, \cdot)$ are all contained in a compact set and the Fourier Transform of $f(\cdot, x_2)$ is compactly supported, 
the two functions $f(\cdot, x_2)$ and $f(x_1,\cdot)$ are analytic, and so is $f$ by Hartogs's Theorem, see \cite[Theorem 1.2.5]{kratz:01}.

For this class of functions, a Logvinenko-Sereda-type Theorem holds and its proof is an adaptation of the arguments used in 
\cite{Kovrijkine-thesis, Kovrijkine-01, EgidiV-16}. However, for the reader's convenience, we repeat the proof here.

\begin{thm}\label{thm:log-ser-strip}
Let $R\in(\RR_+)^{d-1}$ and $f\in L^2(\Omega_R)$ with $\supp h_f\subset J_1\times J_2$ for $J_1, J_2$ as above. 
Set $b=(b_1,\ldots, b_d)$. 
Let $S\subset\RR^d$ be a $(\gamma, a)$-thick set with $a_j\leq 2\pi R_j$ for $j\in\{1,\ldots, d-1\}$. Then, 
\begin{equation} 
\norm{f}{L^2(\Omega_R)}\leq \left(\frac{K^d}{\gamma}\right)^{K a\cdot b+\frac{6d-1}{2}}\norm{f}{L^2(S\cap\Omega_R)},
\end{equation}
where $a\cdot b$ stands for the euclidean inner product in $\RR^d$ and $K>0$ is a universal constant.
\end{thm}

\begin{rmk}\label{rmk:constant}
 Keeping track of the universal constants in the proof of Theorem \ref{thm:log-ser-strip} it is easy to see that $K\geq e$. 
\end{rmk}

Instrumental to the proof of Theorem \ref{thm:log-ser-strip} are the following three lemmas. 
The first one is proved in \cite[Lemma 1]{Kovrijkine-01}, the second one is announced in \cite{Kovrijkine-01} 
and proved in \cite[Lemma 15]{EgidiV-16}, and the third one is a Bernstein inequality for 
$L^2$-functions on $\Omega_R$, where the fact that $J_1,J_2$ are assumed centred at zero is necessary.

\begin{lemma}\label{lemma:1}
Let $z_0\in\RR$ and let $\phi$ be an analytic function on $D(z_0,5):= \{z \in \CC \mid |z-z_0| < 5\}$ such that $\abs{\phi(z_0)}\geq 1$. 
Let $I\subset\RR$ be an interval of unit length 
with $z_0\in I$, and let $A\subset I$ be a measurable set of non-zero measure.
Set $M:=\max_{\abs{z-z_0}\leq 4}\abs{\phi(z)}$. Then
\begin{equation}\label{eq:lemma1}
\sup_{x\in I}\abs{\phi(x)}\leq\left(\frac{12}{\abs{A}}\right)^{2\frac{\log M}{\log 2}}\cdot\sup_{x\in A}\abs{\phi(x)}.
\end{equation}
\end{lemma}

We point out that the proof of the above lemma indirectly gives an estimate of $M$, 
that is $M \geq 2^n$ for some $n\in\NN$ depending on $\phi$. 
It then follows that $M\geq 2$ for all $\phi$ satisfying the assumption of the lemma above.  

\begin{lemma}\label{lemma:level-set-argument}
Let $U\subset\Lambda\subset\RR^d$ be measurable sets with $\abs{\Lambda}=1$ and $\abs{U}>0$.
Let $f\in L^2(\Lambda)$, $C\in [1,\infty)$, and $\alpha\in(0,\infty)$. We define 
\begin{equation*} \label{eq:level-set-1}
{W}:=\left\{x\in \Lambda\;\vert\;\abs{f(x)}<\left(\frac{\abs{U}}{1+C}\right)^\alpha\norm{f}{L^2(\Lambda)}\right\}.
\end{equation*}
and assume that
\begin{equation*}\label{eq:assumption}
\sup_{x\in U}\abs{f(x)}\geq\left(\frac{\abs{U}}{C}\right)^\alpha\norm{f}{L^2(\Lambda)},
\qquad\text{ and }\qquad
\sup_{x\in {W}}\abs{f(x)}\geq\left(\frac{\abs{{W}}}{C}\right)^\alpha\norm{f}{L^2(\Lambda)}.
\end{equation*}
Then, $\abs{{W}}\leq C(1+C)^{-1}\abs{U}$ and
\begin{equation}\label{eq:level-set-2}
\norm{f}{L^2(U)}\geq\left(\frac{\abs{U}}{1+C}\right)^{\alpha+\frac{1}{2}}\norm{f}{L^2(\Lambda)}.
\end{equation}
\end{lemma}

\medskip 

We recall the multi-index notation: Let $\alpha\in\NN_0^d$ and let $b\in\RR^d$, then $b^\alpha:=b_1^{\alpha_1}\cdot\ldots\cdot b_d^{\alpha_d}$, 
and $\abs{\alpha}:=\alpha_1+\ldots+\alpha_d$.
\begin{lemma}\label{lemma:bernstein}
Let $R\in(\RR_+)^{d-1}$, $f\in L^2(\Omega_R)$ be as in \eqref{eq:function}, and set $b=(b_1,\ldots, b_d)$. Then, 
\begin{equation}\label{eq:bernstein-strip} 
\norm{\partial^\alpha f}{L^2(\Omega_R)}\leq (C_B b)^{\alpha}\norm{f}{L^2(\Omega_R)},\qquad \forall\; \alpha\in\NN^d_0,
\end{equation}
where $C_B>1$ is a universal constant.
\end{lemma}

\begin{proof}
We first recall that $f(\cdot, x_2)\colon\TT^{d-1}_R\rightarrow\CC$ is an $L^2(\TT^{d-1}_R)$-function with Fourier frequencies supported in $J_1$ and that 
$f(x_1, \cdot)\colon\RR\rightarrow\CC$ 
is an $L^2(\RR)$-function with Fourier Transform supported in $J_2$. 

It suffices to show the inequality for $\alpha=e_j$, $e_j$ being the vectors of the standard basis of $\RR^d$.
The other cases will then follow iteratively. 

Let $\alpha=e_j$ for a $j\in\{1,\ldots, d-1\}$. By Fubini's Theorem and the Bernstein inequality on the torus (see \cite[Prop. 1.11]{muscalo-schlag:13}) we have 
\begin{multline*}
\int_{\Omega_R} \abs{\partial^{e_j} f(x_1,x_2)}^2 \dd x_1 \dd x_2  =\int_\RR \norm{\partial^{e_j} f(\cdot, x_2)}{L^2(\TT^{d-1}_R)}^2 \dd x_2 \\
\leq (C b)^{2e_j} \int_\RR \norm{f(\cdot,x_2)}{L^2(\TT^{d-1}_R)}^2 \dd x_2 = (C b_j)^2\norm{f}{L^2(\Omega_R)}^2,
\end{multline*}
where $C>1$ is a universal constant. 

Let now $\alpha= e_d$. Using the Bernstein inequality on $\RR$ (see \cite[Chapter 11]{boas:54}), we obtain 
$\norm{\partial^{e_d} f}{L^2(\Omega_R)}\leq \tilde{C} b_d\norm{f}{L^2(\Omega_R)}$, 
for a $\tilde{C}>1$ possibly different from $C$. Therefore, 
for $C_B=\max(C,\tilde{C})$ the claim follows. 
\end{proof}

We are now ready to prove Theorem \ref{thm:log-ser-strip}. 

\begin{proof}[Proof of Theorem \ref{thm:log-ser-strip}] 
\emph{Step 1: Special case.} We first assume $2\pi R_j\geq 1$ for all $j=1,\ldots, d-1$ and $a=(1,\ldots,1)$. We cover $\Omega_R$ with unit cubes, namely, let 
$ \Gamma:=\Big(\ZZ^{d-1}\cap\big( [ 0, \lceil 2\pi R_1\rceil -1)\times\ldots\times[0, \lceil 2\pi R_{d-1}\rceil -1)\big)\Big)\times\ZZ$ so that 
\[
\Omega_R=\TT^{d-1}_R\times \RR\subset \bigcup_{j\in\Gamma}\Lambda_j,\qquad \Lambda_j:=[0,1]^d +j.
\]
Consequently,
\begin{equation}\label{eq:sum-norm}
\sum_{j\in\Gamma}\norm{f}{L^2(\Lambda_j)}^2\leq 2^{d-1}\norm{f}{L^2(\Omega_R)}^2,\qquad \forall \ f\in L^2(\Omega_R). 
\end{equation}
To ease the notation, we will write $\Lambda$ instead of $\Lambda_j$ and we will denote all universal constants by $C$ allowing them to change from line to line. 

\medskip 

\emph{Step 2: Local estimate.} We now aim at obtaining a local estimate for the $L^2$-norm of $f$ on $\Lambda$ and on $ S\cap\Lambda$ 
using a dimension reduction argument and Lemma \ref{lemma:1}. 

We first prove that given $y\in \Lambda$ there exists a line segment $I:=I(S,y)\subset\Lambda$ depending on $S$ and $y$ such that 
\[
y\in I \qquad\text{and}\qquad \frac{\abs{S\cap I}}{\abs{I}}\geq \frac{\abs{S\cap\Lambda}}{C^d}
\]
for some constant $C>1$. Indeed, set $\sigma_{d-1}=\abs{\mathbb{S}^{d-1}}$. Then, by spherical coordinates
\begin{equation*}\label{eq:eq:13}
\abs{S\cap \Lambda}=\int_{S\cap\Lambda} \dd x=\int_{\abs{\xi}=1}\int_{0}^{\infty}\chi_{S\cap \Lambda}(y+r\xi)r^{d-1}\dd r \dd\sigma(\xi),
\end{equation*}
and there exists a point $\eta\in\mathbb{S}^{d-1}$ such that
\begin{equation}\label{eq:*}
\abs{S\cap \Lambda}\leq \sigma_{d-1}\int_0^{\infty}\chi_{S\cap \Lambda}(y+r\eta)r^{d-1}\dd r.
\end{equation}
Let $I$ be the longest line segment in $\Lambda$ starting at $y$ in the direction $\eta$, i.e.
\begin{equation}\label{eq:line-segment-2}
I=\{x\in \Lambda\quad\vert\quad x=y+r\eta,\quad r\geq 0\}.
\end{equation}	
The estimate $r\leq \sqrt{d}$ and \eqref{eq:*} yield
\begin{equation*}\label{eq:eq:14}	
\abs{S\cap \Lambda}\leq \sigma_{d-1} d^{(d-1)/2}\int_0^{\infty}\chi_{S\cap I}(y+r\eta)\dd r=\sigma_{d-1} d^{(d-1)/2}\abs{S\cap I}.
\end{equation*}
Since $\sigma_{d-1}$ behaves as $\frac{1}{\sqrt{(d-1)\pi}}\left(\frac{2\pi e}{d-1}\right)^{(d-1)/2}$ when $d\to\infty$, there exists 
a constant $C>1$ so that $\sigma_{d-1}d^{d/2}\leq C^d$. 
This fact together with \eqref{eq:*} and the inequality $\abs{I}\leq d^{1/2}$ yields 
\begin{equation}\label{eq:line-segment}
\frac{\abs{S\cap I}}{\abs{I}}\geq \frac{\abs{S\cap \Lambda}}{\sigma_{d-1}d^{d/2}}\geq \frac{\abs{S\cap \Lambda}}{C^d},
\end{equation}
for a $C>1$ universal constant.

Let now $y_0\in \Lambda$ be a point such that $\abs{f(y_0)}\geq\norm{f}{L^2(\Lambda)}$, 
e.g. the maximum of $f$ in $\Lambda$, and define $F\colon\CC\rightarrow\CC$
by $F(w)=\norm{f}{L^2(\Lambda)}^{-1}f(y_0+w\abs{I_0}\eta)$, where $I_0:=I(S,y_0)$ and $\eta$ are as in \eqref{eq:line-segment-2}.
We apply Lemma \ref{lemma:1} to $F$, $[0,1],$ and $A:=\{t\in[0,1]\;\vert\;y_0+t\abs{I_0}\eta\in S\cap I_0\}$,  
note that $\abs{A}=\frac{\abs{S\cap I_0}}{\abs{I_0}}$. Then 
\begin{align*}\label{eq:eq:15}
\sup_{x\in S\cap \Lambda}\abs{f(x)}& \geq\sup_{x\in S\cap I_0}\abs{f(x)}=\norm{f}{L^2(\Lambda)}\sup_{t\in A}\abs{F(t)} \\
&\geq\norm{f}{L^2(\Lambda)}\Big(\frac{\abs{A}}{12}\Big)^{\frac{2\log M}{\log 2}}\sup_{t\in [0,1]}\abs{F(t)}\\
&=\Big(\frac{\abs{A}}{12}\Big)^{\frac{2\log M}{\log 2}}\sup_{t\in[0,1]}\abs{f(y_0+t\abs{I_0}\eta)}\\
&\geq\Big(\frac{\abs{A}}{12}\Big)^{\frac{2\log M}{\log 2}}\abs{f(y_0)}
\geq\Big(\frac{\abs{S\cap I_0}}{12\abs{I_0}}\Big)^{\frac{2\log M}{\log 2}}\norm{f}{L^2(\Lambda)}\\
&\geq\Big(\frac{\abs{S\cap \Lambda}}{C^d}\Big)^{\frac{2\log M}{\log 2}}\norm{f}{L^2(\Lambda)},
\end{align*}
where in the last step we used \eqref{eq:line-segment} and where $M=\max_{\abs{w}\leq 4}\abs{F(w)}$.
Similarly, for
\[
V=\left\{x\in \Lambda \ \Big\vert \ \abs{f(x)}<\left(\frac{\abs{S \cap \Lambda}}{1+C^d}\right)^{\frac{2\log M}{\log 2}}\norm{f}{L^2(\Lambda)}\right\}
\]
we obtain
\begin{equation}\label{eq:V}
\sup_{x\in {V}}\abs{f(x)}
\geq\left(\frac{\abs{{V}}}{C^d}\right)^{\frac{2\log M}{\log 2}}\norm{f}{L^2(\Lambda)},
\end{equation}
using a possibly different line segment $I({V},y_0)\subset \Lambda$ containing $y_0$
and satisfying a proportionality relation analogous to \eqref{eq:line-segment} with $S$ replaced by $V$.

Lemma \ref{lemma:level-set-argument} applied with $U=S\cap \Lambda$ and $\alpha=2\log M/\log 2$ gives
\begin{equation}\label{eq:1}
\begin{split}
\norm{f}{L^2(S\cap \Lambda)}&\geq\left(\frac{\abs{S\cap\Lambda}}{1+C^d}
\right)^{\frac{2\log M}{\log 2}+\frac{1}{2}}\norm{f}{L^2(\Lambda)}
\geq\left(\frac{\gamma}{C^d}\right)^{\frac{2\log M}{\log 2}+\frac{1}{2}}\norm{f}{L^2(\Lambda)}.
\end{split}
\end{equation}

We are now left with estimating $M=\max_{\abs{w}\leq 4}\abs{F(w)}$,
which depends on the particular cube $\Lambda=[0,1]^d+j$ under consideration. 
It turns out it is enough to estimate the maximum on a special class of cubes.

\medskip 

\emph{Step 3: Good and bad cubes.} We say that $\Lambda$ is a good cube if 
for all multi-indices $\alpha\in\NN^d$
\begin{equation}\label{eq:bad}
\norm{\partial^{\alpha}f}{L^2(\Lambda)}< 2^{\frac{2d-1}{2}}(3C_B b)^{\alpha}\norm{f}{L^2(\Lambda)},
\end{equation}
where $C_B$ is the constant in Lemma \ref{lemma:bernstein}. We call $\Lambda$ bad otherwise. 
This estimate can be regarded as a local Bernstein inequality. 

As a consequence we obtain
\begin{equation}\label{eq:bad-cubes}
\norm{f}{L^2\left(\bigcup\limits_{\Lambda\text{ bad}}\Lambda\right)}^2\leq\frac{1}{2}\norm{f}{L^2(\Omega_L)}^2,
\end{equation}
and therefore there exist good cubes. In fact, 
using the definition of bad cubes, Ineq. \eqref{eq:sum-norm}, and Lemma \ref{lemma:bernstein}, we have
\begin{align*}\label{eq:eq:10}
\norm{f}{L^2\left(\bigcup\limits_{\Lambda\text{ bad}}\Lambda\right)}^2 & 
\leq\sum_{\alpha\in\NN^d}\sum_{\Lambda\text{ bad}}\frac{1}{2^{2d-1}(3 C_B b)^{2\alpha}}\norm{\partial^\alpha f}{L^2(\Lambda)}^2\\
&\leq\sum_{\alpha\in\NN^d}\frac{2^{d-1}}{2^{2d-1}(3 C_B b)^{2\alpha}}\norm{\partial^\alpha f}{L^2(\Omega_R)}^2\\
&\leq\sum_{\alpha\in\NN^d}\frac{1}{2^{d}3^{2\abs{\alpha}}}\norm{f}{L^2(\Omega_R)}^2\\
&=\frac{1}{2^d}\left(\frac{1}{\left(1-\frac{1}{9}\right)^d}-1\right)\norm{f}{L^2(\Omega_R)}^2\leq \frac{1}{2}\norm{f}{L^2(\Omega_R)}^2. 
\end{align*}
We now claim that for a good cube $\Lambda$ there exists a point $x\in \Lambda$ such that 
\begin{equation}\label{eq:claim}
\abs{\partial^\alpha f(x)}\leq {2^{(3d-1)/2}}(9 C_B b)^{\alpha}\norm{f}{L^2(\Lambda)}\qquad\forall\;\alpha\in\NN_0^d.
\end{equation}
Indeed, arguing by contradiction, assume that for every $x\in\Lambda$, with $\Lambda$ being a good cube,
there exists $\alpha(x)\in\NN_0^d$ such that
\begin{equation*}
\abs{\partial^{\alpha(x)} f(x)}> 2^{(3d-1)/2} (9 C_B b)^{\alpha(x)} \norm{f}{L^2(\Lambda)}.
\end{equation*}
To get rid of the $x$-dependence in $\alpha(x)$ we divide and sum over all multi-indices, so that
\begin{equation*}
\sum_{\alpha\in\NN_0^d}  \frac{\abs{\partial^\alpha f(x)}^2}{ 2^{3d-1} (9 C_B b)^{2\alpha} }
\geq\frac{\abs{\partial^{\alpha(x)} f(x)}^2}{ 2^{3d-1} (9 C_B b)^{2\alpha(x)} }
> \norm{f}{L^2(\Lambda)}^2.
\end{equation*}
Then, integration over $\Lambda$ and the definition of good cubes yield
\begin{align*}\label{eq:eq:11}
2^{d-1}\norm{f}{L^2(\Lambda)}^2 & \leq\sum_{\alpha\in\NN_0^d}\frac{1}{2^{2d}(9 C_B b)^{2\alpha}}\norm{\partial^\alpha f}{L^2(\Lambda)}^2\\
& \leq\sum_{\alpha\in\NN_0^d}\frac{1}{2}\left(\frac{1}{9}\right)^{\abs{\alpha}}\norm{f}{L^2(\Lambda)}^2
=\frac{1}{2}\left(\frac{9}{8}\right)^d\norm{f}{L^2(\Lambda)}^2,
\end{align*}
and, consequently,
\begin{equation*}\label{eq:eq:12}
\norm{f}{L^2(\Lambda)}^2\leq\left(\frac{9}{16}\right)^d\norm{f}{L^2(\Lambda)}^2<\norm{f}{L^2(\Lambda)}^2,
\end{equation*}
giving the desired contradiction. 

\medskip

\emph{Step 4: Conclusion.} Let now $\Lambda=[0,1]^d+j$, $j\in\Gamma$ (see Step 1), be a good cube and assume it is centred at some point $s\in\RR^d$, 
i.e. we have  
\begin{equation*} \label{eq:lambda-cube}
\Lambda=[s_1-1/2,s_1+1/2]\times\ldots\times[s_d-1/2,s_d+1/2],
\end{equation*}
and let $D(z_0,r)=\{z\in\CC\;\vert\;\abs{z-z_0}< r\}$ for $z_0\in\CC$.

Let now $y_0\in\Lambda$, and $\eta$ and $I_0$ chosen as in Step 2. 
We have that $\abs{\abs{I_0}\eta_i}\leq 1$ for all $i\in\{1,\ldots,d\}$. 
Therefore, if $w\in D(0,4)$, we obtain  
$y+w\abs{I_0}\eta\in\widetilde{D}:= D(s_1,4+1/2)\times\ldots\times D(s_d,4+1/2)$.
Let now $x\in\Lambda$ as in \eqref{eq:claim}, then $\widetilde{D}\subset D(x_1,5)\times\ldots\times D(x_d,5)$. 
For any $z\in \widetilde{D}$ Taylor expansion yields
\begin{equation}\label{eq:taylor}
\begin{split}
\abs{f(z)}\leq\sum_{\alpha\in\NN_0^d}\frac{\abs{\partial^\alpha f(x)}}{\alpha!}\abs{z-x}^{\abs{\alpha}}
&\leq\sum_{\alpha\in\NN_0^d}{2^{(3d-1)/2}}(9 C_B b)^{\alpha}5^{\abs{\alpha}}\frac{1}{\alpha!}\norm{f}{L^2(\Lambda)}\\
&={2^{(3d-1)/2}}\exp\left(45 C_B \norm{b}{1}\right)\norm{f}{L^2(\Lambda)},
\end{split}
\end{equation}
with $\norm{b}{1}= \sum_{j=1}^d b_j$ and where Ineq. \eqref{eq:claim} is used in the second step.

We are now able to bound the maximum of $F(w)=\norm{f}{L^2(\Lambda)}^{-1}f(y_0+w\abs{I_0}\eta)$
associated with the good cube $\Lambda$. By \eqref{eq:taylor} we estimate
\begin{equation*}\label{eq:M-bound}
M=\max_{\abs{w}\leq 4}\abs{F(w)}
\leq\norm{f}{L^2(\Lambda)}^{-1}\max_{z\in\widetilde{D} }\abs{f(z)}\leq {2^{(3d-1)/2}}\exp\left(45 C_B\norm{b}{1}\right).
\end{equation*}

Consequently, $\log M\leq\left(\frac{3d-1}{2}\right)\log 2+45 C_B\norm{b}{1}$ and
\begin{equation}\label{eq:M-bound-2}
\frac{2\log M}{\log 2}+\frac{1}{2} 
\leq \frac{6d -1}{2}+\frac{90 C_B}{\log 2}\norm{b}{1}.
\end{equation}

Substituting \eqref{eq:M-bound-2} into \eqref{eq:1}, summing over all good cubes $\Lambda$, and using \eqref{eq:bad-cubes} we have
\begin{align}\label{eq:semi-final}
\norm{f}{L^2(S\cap\Omega_R)} & \geq \norm{f}{L^2(S\cap(\bigcup_{\Lambda\text{ good}}\Lambda)}\notag\\ 
& \geq \frac{1}{2^{d-1}}\left(\frac{\gamma}{C^d}\right)^{\frac{6d-1}{2}+\frac{90 C_B}{\log 2}\norm{b}{1}}\norm{f}{L^2(\bigcup_{\Lambda \text{ good}}\Lambda)}\notag\\
& \geq \frac{1}{2^{d}}\left(\frac{\gamma}{C^d}\right)^{\frac{6d-1}{2}+\frac{90 C_B}{\log 2}\norm{b}{1}}\norm{f}{L^2(\Omega_R)}\notag\\
& \geq \left(\frac{\gamma}{K^d}\right)^{\frac{6d-1}{2}+K\norm{b}{1}}\norm{f}{L^2(\Omega_R)}
\end{align}
for $K=\max\left(\frac{90 C_B}{\log 2}, (2 C)^d\right)$. This concludes the proof for $a=(1,\ldots,1)$ and $2\pi R_j\geq 1$ for $j=1,\ldots, d-1$.

\medskip 

\emph{Step 5: General case.} Let us now assume that $R\in(\RR_+)^{d-1}$, the vector $a=(a_1,\ldots,a_d)$ has components $a_j\leq 2\pi R_j$ for all $j=1,\ldots,d-1$, 
$S$ is a $(\gamma, a)$-thick set, and $f\in L^2(\Omega_R)$ is as in \eqref{eq:function}. 

We define the transformation map $T(x_1,\ldots,x_{d})=(a_1x_1,\ldots,a_d x_d)$ for all $x\in \RR^d$.
In particular, $T(\Omega_{R/a})=\Omega_R$ for $\Omega_{R/a}:=[0,\frac{2\pi R_1}{a_1}]\times\ldots\times[0,\frac{2\pi R_{d-1}}{a_{d-1}}]\times\RR$, 
and $G:=T^{-1}(S)$ is $(\gamma, 1)$-thick.
Further, for the function $g:=f\circ T:\Omega_{R/a}\rightarrow\CC$ we have  
\begin{equation*}\label{eq:support-scaling}
\supp h_g\subset \left[-\frac{a_1 b_1}{2},\frac{a_1 b_1}{2}\right]\times\ldots\times\left[-\frac{ a_{d} b_{d}}{2},\frac{a_{d} b_{d}}{2}\right],\\
\end{equation*}
\begin{equation*}\label{eq:p-norms}
\left(\prod_{j=1}^d a_j\right)\norm{g}{L^2(\Omega_{R/a})}^2=\norm{f}{L^2(\Omega_R)}^2,
\quad\text{ and }\quad
\left(\prod_{j=1}^d a_j\right)\norm{g}{L^2(G\cap \Omega_{R/a})}^2=\norm{f}{L^2(S\cap \Omega_R)}^2.
\end{equation*}
Therefore, applying \eqref{eq:semi-final} to $g$, $G$ and $\Omega_{R/a}$, and using the scaling relations above we conclude
\begin{multline*}\label{eq:eq:18}
\norm{f}{L^2(S\cap\Omega_R)}^2=\left(\prod_{j=1}^d a_j\right)\norm{g}{L^2(G\cap \Omega_{R/a})}^2
\geq\left(\prod_{j=1}^d a_j\right)\left(\frac{\gamma}{K^d}\right)^{6d-1+2K a\cdot b}\norm{g}{L^2(\Omega_{R/a})}^2\\
=\left(\frac{\gamma}{K^d}\right)^{6d-1+2 K a\cdot b}\norm{f}{L^2(\Omega_R)}^2.
\end{multline*}
\end{proof}

\section{Proof of Proposition \ref{prop:sufficiency}}\label{sec:sufficiency}

\subsection{Proof of the null-controllability}
The proof of Proposition \ref{prop:sufficiency} is an application of Theorem \ref{thm:observability-B-PS}, 
where as orthogonal projection we choose the spectral projection of the (minus) Laplacian 
on $\Omega_L$. With this choice, the dissipation estimate \eqref{eq:dissipation-general} follows automatically, 
while the spectral inequality \eqref{eq:spectral-inequality-general} is a consequence of Theorem \ref{thm:log-ser-strip}.  

We treat Dirichlet and Neumann boundary conditions simultaneously denoting by $-\Delta^\bullet$, $\bullet\in \{D,N\}$, 
the (minus) Laplacian on $\Omega_L$ with Dirichlet or Neumann boundary conditions.

Using the language of tensor products, we recall that 
\[
L^2(\Omega_L)=L^2(\TT^{d-1}_L\times \RR)\cong L^2(\TT^{d-1}_L)\otimes L^2(\RR)
\]
and that the (minus) Laplacian on $\Omega_L$ can be written as 
\[
-\Delta^\bullet=(-\Delta_{1}^\bullet)\otimes I_2+ I_{1}\otimes (-\Delta_2),
\]
where $\Delta_1^\bullet$ is the Laplacian on $\TT^{d-1}_L$ with Dirichlet or Neumann boundary 
conditions, $\Delta_2$ is the Laplacian on $\RR$, and $I_1, I_2$ are the identity operators on $\TT^{d-1}_L$ and $\RR$, respectively. 
Moreover, this operator is self-adjoint \cite[Theorem 7.23 and Ex. 7.17(a)]{schmuedgen:12}.

Let now $E\geq 1$ and let $\pi_E(-\Delta^\bullet)$ be the spectral projection of $-\Delta^\bullet$ associated to the interval $(-\infty, E]$.

We first show the dissipation inequality. We have $1-\pi_E(-\Delta^\bullet)=\pi_{(E,+\infty)}(-\Delta^\bullet)$ and
by spectral calculus $e^{2t\Delta^\bullet}\pi_{(E,+\infty)}(-\Delta^\bullet) \leq e^{-2tE}\pi_{(E,+\infty)}(-\Delta^\bullet)$ in the sense of quadratic forms.
This yields the dissipation estimate
\begin{equation}\label{eq:dissipation-estimate}
\begin{aligned}
\norm{\pi_{(E,+\infty)}(-\Delta^\bullet)(e^{t\Delta^\bullet}f)}{L^2(\Omega_L)}^2
& =\langle \pi_{(E,+\infty)}(-\Delta^\bullet) f, \pi_{(E,+\infty)}(-\Delta^\bullet) e^{2t\Delta^\bullet}f\rangle_{L^2(\Omega_L)}\\
& \leqslant e^{-2tE}\norm{\pi_{(E,+\infty)}(-\Delta^\bullet)f}{L^2(\Omega_L)}^2\\
& \leq e^{-2tE}\norm{f}{L^2(\Omega_L)}^2,
\end{aligned}
\end{equation}
which implies \eqref{eq:dissipation-general} with $c_2=\eta_2=m=1$.

We now derive the spectral inequality.  
Recall that the operator $-\Delta^\bullet_1$ has purely discrete spectrum and that its eigenvalues and eigenfunctions with Dirichlet and Neumann conditions are: 
\begin{align*}
\lambda_n=\frac{\norm{n}{2}^2}{(2L)^2}, \qquad \phi_n^D(x)= & \left(\frac{1}{\pi L}\right)^{\frac{d-1}{2}}\prod_{j=1}^{d-1}\sin\left(\frac{n_j x_j}{2 L}\right) \quad n\in\NN^{d-1},\\
\lambda_n=\frac{\norm{n}{2}^2}{(2L)^2}, \qquad \phi_n^N(x)= & \left(\frac{1}{\pi L}\right)^{\frac{d-1}{2}}\prod_{j=1}^{d-1}\cos\left(\frac{n_j x_j}{2 L}\right)\quad n\in\NN_0^{d-1},
\end{align*}
respectively, where $\norm{n}{2}^2=\sum_{j=1}^{d-1} \abs{n_j}^2$.  
To further ease the notation we set $\NN_D=\NN$ in case of Dirichlet boundary conditions, and $\NN_N= \NN_0$ in case of Neumann boundary conditions. 

Since the above eigenfunctions form an orthonormal basis of $L^2(\TT^{d-1}_L)$, any $f\in L^2(\Omega_L)$ can be expanded as
\begin{align*} 
f(x_1,& x_2)=\sum_{n\in\NN_\bullet^{d-1}} \left(\int_{\RR} h_f^\bullet\left(\frac{n}{2L},\xi\right) e^{i x_2\xi} \dd \xi \right) \phi_n^\bullet(x_1), 
\end{align*}
where $h^\bullet_f$ is defined as 
\begin{align*}
h_f^\bullet:\Big(\frac{1}{2L}\NN_\bullet\Big)^{d-1}&\times\RR\rightarrow \CC, \quad 
h_f^\bullet\Big(\frac{n}{2L}, \xi\Big)= \frac{1}{\sqrt{2\pi}}\int_{\TT^{d-1}_L}\int_\RR f(s,t)\phi_n^\bullet(s) e^{-it\xi} \dd t \dd s.
\end{align*}

Let now $\pi_E(-\Delta_2)$ be the spectral projection of $-\Delta_2$ associated to the interval $(-\infty, E]$ and let $P_{\lambda_n}^\bullet$ 
be the projection on the kernel of $-\Delta_1^\bullet-\lambda_n$.
Since $-\Delta^\bullet_1$ has purely discrete spectrum, from \cite[Theorem 8.34]{weidmann:80} we infer 
\begin{equation*} 
\pi_E(-\Delta^\bullet)  = \sum_{n\in \NN_\bullet^{d-1}} P_{\lambda_n}^\bullet \otimes \pi_{E-\lambda_n}(-\Delta_2) 
= \sum_{\lambda_n\leq E} P_{\lambda_n}^\bullet \otimes \pi_{E-\lambda_n}(-\Delta_2),
\end{equation*}
and since $f$ is represented by a linear combination of products of type $g_1(x_1) g_2(x_2)$, we obtain 
\begin{equation*} 
\pi_E(-\Delta^\bullet)f (x_1,x_2) = \sum_{\lambda_n\leq E} \left(\int_{\{\xi^2\leq E-\lambda_n\}}h_f^\bullet\left(\frac{n}{2L},\xi\right) e^{i x_2\xi} \dd \xi\right)\phi_n^\bullet(x_1).
\end{equation*}

Since the eigenfunctions $\phi^\bullet_n$ have no finite Fourier series with respect to $\TT^{d-1}_L$, the expansion of $\pi_E(-\Delta^\bullet)f$ by Fourier Analysis, as done at the beginning 
of Section \ref{sec:spectal-inequality}, gives a function $h_{\pi_E(-\Delta^\bullet)f}$ with no compact support, and so Theorem \ref{thm:log-ser-strip} is not directly applicable. 
However, the $\phi_n^\bullet$ have finite Fourier series with respect to $\TT^{d-1}_{2L}$. We therefore extend $\pi_E(-\Delta^\bullet)f$ to the strip $\Omega_{2L}=\TT^{d-1}_{2L}\times\RR$ by antisymmetric and symmetric reflections 
with respect to the boundary of $\Omega_L$, in case of Dirichlet and Neumann boundary conditions, respectively. The extended function $F$ is then given by 
\begin{multline*}
F(x_1,x_2) = \sum_{\lambda_n\leq E} \left(\int_{\{\xi^2\leq E-\lambda_n\}}h_f^\bullet\left(\frac{n}{2L},\xi\right) e^{i x_2\xi} \dd \xi\right)\phi_n^\bullet(x_1),\\ (x_1,x_2)\in\Omega_{2L}.
\end{multline*}
On $\TT^{d-1}_{2L}$, the Fourier series of $F$ in the variable $x_1\in\TT^{d-1}_{2L}$ is finite with Fourier frequencies all contained in 
$[-\sqrt{E},\sqrt{E}]^{d-1}$. Consequently, the function $h_{F}$ has support contained in $[-\sqrt{E},\sqrt{E}]^{d}$.

Let now $S$ be $(\gamma,a)$-thick. We define
\begin{align*}
S^{(0)} &= S\cap \Omega_L, \\
S^{(1)} &= S^{(0)} \cup \{(-x_1,x_2,\ldots,x_d)\mid (x_1,x_2,\ldots,x_d)\in   S^{(0)} \}, \\
S^{(2)} &= S^{(1)} \cup \{(x_1,-x_2,x_3,\ldots,x_d)\mid (x_1,x_2,\ldots,x_d)\in   S^{(1)} \},\\
\vdots  &= \vdots \\
S^{(d-1)} &= S^{(d-2)} \cup \{(x_1,\ldots,-x_{d-1},x_d)\mid (x_1,x_2,\ldots,x_d)\in   S^{(d-2)} \},
\end{align*}
and extend $S^{(d-1)}$ periodically to $ \hat{S}=  \bigcup_{\kappa \in (4\pi L\ZZ)^{d-1}\times\{0\}} \left(\kappa +S^{(d-1)}\right)$.
With an argument similar to the proof of Lemma \ref{lemma:thickness} or to the one in \cite[Section 5]{EV-17}, it is easy to prove that $\hat{S}$ is a $(\gamma/2^d, 2a)$-thick set in $\RR^d$. 

By reflection symmetry of $F$ and $\hat S$, for all  $\kappa\in \Upsilon:=\{0, 2\pi L\}^{d-1}\times\{0\}$ we have
\[
\norm{F}{L^2( S \cap \Omega_L) }^2=\norm{F}{L^2(\hat S \cap \Omega_L) }^2
=\norm{F}{L^2(\hat S \cap (\kappa +\Omega_L)) }^2,
\] 
and 
\[\norm{F}{L^2(\Omega_L) }^2
=\norm{F}{L^2(\kappa +\Omega_L)}^2.\]
Consequently
\begin{align*}
\norm{F}{L^2(\Omega_{2L})}^2 & 
=\sum_{\kappa\in \Upsilon}\norm{F}{L^2(\kappa+\Omega_L)}^2=2^{d-1}\norm{\pi_E(-\Delta^\bullet) f}{L^2(\Omega_L)}^2,\\
\norm{F}{L^2(\hat{S}\cap \Omega_{2L})}^2 & 
=\sum_{\kappa\in \Upsilon}\norm{F}{L^2(\hat{S} \cap (\kappa+\Omega_L))}^2
=2^{d-1}\norm{\pi_E(-\Delta^\bullet) f}{L^2(S\cap \Omega_L)}^2.
\end{align*}
Now, Theorem \ref{thm:log-ser-strip} applied to $F$ and $\hat{S}$ yields
\begin{align*}
\norm{\pi_E(-\Delta^\bullet) f}{L^2(\Omega_L)}^2 & = 2^{-d+1}\norm{F}{L^2(\Omega_{2L})}^2\\
& \leq 2^{-d+1}\left(\frac{(2 K)^d}{\gamma}\right)^{ 8 K\sqrt{E} \norm{a}{1}+6d-1}\norm{F}{L^2(\hat{S}\cap \Omega_{2L})}^2\\
& = \left(\frac{(2 K)^d}{\gamma}\right)^{ 8 K\sqrt{E} \norm{a}{1}+6d-1} \norm{\pi_E(-\Delta^\bullet) f}{L^2(S\cap \Omega_{L})}^2\\
& \leq \left(\frac{(2 K)^d}{\gamma}\right)^{8 K\sqrt{E}( \norm{a}{1}+d)} \norm{\pi_E(-\Delta^\bullet) f}{L^2(S\cap \Omega_{L})}^2,
\end{align*}
which fullfils the spectral inequality \eqref{eq:spectral-inequality-general} with $\eta_1=\frac{1}{2}$ and 
\[
c_1= 4 K( \norm{a}{1}+d)\log \left(\frac{(2 K)^d}{\gamma}\right) \geq 3e,
\]
where inequality holds since $K\geq e$, see Remark \ref{rmk:constant}.

Therefore, Theorem \ref{thm:observability-B-PS} implies the controllability of the adjoint system and consequently 
the null-controllability of system \eqref{eq:heat_equation}.

\subsection{Control cost}
Let $c_1, c_2, \eta_1, \eta_2, m$ be the constants in the previous subsection. 
From the proof of Theorem \ref{thm:observability-B-PS} in \cite[Appendix 8.3]{B-PS:17} we infer
\begin{equation}
C_2= 144 c_1^2 \quad \text{ and } \quad C_1=\exp\left(\frac{2 C_2}{2\tau_0}\right),
\end{equation}
where $\tau_0$ is such that for all $0<\tau <\tau_0$
the following inequalities are fulfilled:
\begin{align}
&\tau<2^{5/2}3c_1,\label{eq:**}\\
& h_1(\tau):=\frac{1}{\tau}\exp\left(-\frac{2^3 3 c_1^2}{\tau}\right)\leq \frac{1}{4},\label{eq:2}\\
& h_2(\tau):=\frac{1}{\tau}\exp\left(\frac{2^4 (3 c_1)^2}{\tau}\right)\geq 1.\label{eq:3}
\end{align}

We may choose $\tau_0= 2^{5/2} 3 c_1 $. 
Eq.~\eqref{eq:2} is fulfilled for any $0<\tau\leq \tau_1=2^3 3 c_1^2 $, since for positive $\tau$ the function $h_1$
has a maximum in $\tau_1$. Hence, for all $0<\tau\leq \tau_1$ and our choice of $c_1$ we have 
\[
h_1(\tau)\leq \frac{1}{\tau_1}\exp\left(-\frac{2^3 3 c_1^2}{\tau_1}\right)=
\frac{1}{2^3 3 c_1^2 e}\leq \frac{1}{2^3 3 (3 e)^2 e}\leq \frac{1}{216 e^3}< \frac{1}{4}.
\]

Finally, Eq.~\eqref{eq:3} is fulfilled for all $\tau\leq 2^{5/2} 3 c_1 $ since $h_2$ is a decreasing function. 
Indeed, using $\exp(x)\geq \frac{x^2}{2}$ and $c_1\geq 3 e$, we obtain
\[
h_2(\tau)\geq 
\frac{1}{2^{5/2} 3 c_1}\exp(2^{3/2}3c_1)
\geq \frac{2^{2} (3 c_1)^2}{2^{5/2} 3 c_1}\geq 1.
\]

Therefore, $\tau_0=\min(2^{5/2} 3 c_1, 2^3 3 c_1^2)=2^{5/2} 3 c_1$ and, consequently,  
\begin{align*}
C_2=\exp\left(\frac{2 C_1}{2\tau_0}\right)=\exp(6\sqrt{2} c_1).
\end{align*}
Therefore, we obtain the control cost estimate $C_T\leq \sqrt{C_1}\exp\left(\frac{C_2}{2T}\right)$ where 
\[
\sqrt{C_1}=\left(\frac{(2K)^d}{\gamma}\right)^{12\sqrt{2}K(\norm{a}{1}+d)}, \qquad C_2=144(4K)^2(\norm{a}{1}+d)^2\log^2\left( \frac{(2K)^d}{\gamma}\right),
\]
as claimed in Proposition \ref{prop:sufficiency}.

\section{Necessity condition}\label{sec:necessity}

We now prove the necessity condition in Theorem \ref{thm:main}, i.e. the implication $(ii)\Rightarrow (i)$.
The proof uses a contradiction argument and heat kernel estimates, i.e. estimates on the 
integral kernel for the semigroup $e^{t\Delta}$. In what follows, $K_{\Omega_L}(t,x,y)$ denotes the heat kernel of $\Omega_L$. 
%

Let us assume that the control set $\omega$ is not thick with respect to $\Omega_L$.  
Then, for all $\gamma>0$ and for all $a\in(\RR_+)^d$
there exists a hyperrectangle $Q_{\gamma,a}$ centred at some point $x_{\gamma,a}\in\Omega_L$ 
with sides of length $a_1,\ldots, a_d$ such that  
\[
\abs{\omega\cap Q_{\gamma,a}}<\gamma\abs{Q_{\gamma,a}}.
\]
Let now $n\in\NN$ and choose $\gamma=1/n^2$ and $a=(2\pi L,\ldots, 2\pi L, n)$ to obtain a 
sequence of hyperrectangles $Q_n\subset \Omega_L$ centred at some point $x_n$ such that 
\begin{equation}\label{eq:cubes}
\abs{\omega\cap Q_n}<(2\pi L)^{d-1} n^{-1}.
\end{equation}  
Due to the choice of the parameter $a$ and the fact that $Q_n\subset\Omega_L$, we have $x_n=(\pi L,\ldots, \pi L,x_{n,d})$ for some $x_{n,d}\in\RR$. 

We first treat the case of Dirichlet boundary conditions and we aim at constructing a sequence of functions 
which does not satisfy the observability estimate \eqref{eq:observability-def}. We consider the initial data $g_n(x)= K_{\Omega_L}(1,x,x_n)$, so that 
$g_n(t,x)=e^{t\Delta}g_n(x)=K_{\Omega_L}(1+t,x,x_n)$ is solution to the adjoint system \eqref{eq:heat_equation-adjoint}. 

Let now $W$ be a $d$-dimensional cube in $\Omega_L$ with sides of length $\pi L$ and centred at $x_n$. 
For the heat kernel on $W$ and $\Omega_L$ the following estimate holds (see \cite[Thm. 2.1.4 and Thm. 2.1.6]{davies:89})
\begin{equation}\label{eq:min-max}
K_{\Omega_L}(t,x,x_n) \geq K_W(t,x,x_n)=\sum_{k\in\NN^d} e^{-t\eta_k}\psi_k(x)\psi_k(x_n)\qquad \forall \ t>0, \quad \forall \ x\in W,
\end{equation}
where 
\[
\eta_k=\frac{\norm{k}{2}^2}{ L^2}\quad\text{ and }\quad
\psi_k(x)=\left(\frac{2}{\pi L}\right)^{d/2}\prod_{j=1}^{d}\sin\left(\frac{ k_j}{ L}\left(x_j-x_{0,j}+\frac{\pi L}{2}\right)\right),\quad k\in\NN^d
\] 
are the eigenvalues and corresponding eigenfunctions of the Dirichlet Laplacian on $W$. Therefore, we obtain 
\begin{align*}
\int_{\Omega_L}\abs{g_n(T,x)}^2 \dd x & = \int_{\Omega_L}\abs{K_{\Omega_L}(1+T,x,x_n)}^2 \dd x \\
& \geq \int_W\abs{K_W(1+T,x,x_n)}^2\dd x \\
& \geq e^{-2(1+T)\eta_{(1,\ldots,1)}}\abs{\psi_{(1,\ldots, 1)}(x_n)}^2\\
& = \left(\frac{2}{\pi L}\right)^{d}\exp\left(-\frac{2(1+T)d}{ L^2}\right)>0,
\end{align*}
i.e. the left hand side of \eqref{eq:observability-def} is bounded from below by a positive constant for all $n\in\NN$.

We now show that the right hand side of \eqref{eq:observability-def} converges to zero as $n\to+\infty$. 
For this purpose we use the upper bound (see \cite[Cor. 3.2.8]{davies:89})
\[
K_{\Omega_L}(t,x,y)\leq \frac{c}{t^{d/2}}\exp\left(-\frac{\norm{x-y}{2}^2}{6t}\right)\qquad \forall \ t>0 \quad \forall \ x,y\in\Omega_L,
\]
for $c$ a positive constant. 

Then, using the change of variable $y=x-x_n$, the monotonicity of the exponential in $t$, and the estimate $e^{-x}\leq 1$ for $x\geq 0$, we calculate
\begin{align*}
\int_0^T \int_\omega \abs{g_n(t,x)}^2 \dd x\dd t & =\int_0^T \int_\omega \abs{K_{\Omega_L}(1+t,x,x_n)}^2 \dd x\dd t\\  
& \leq \int_0^T\int_\omega  \frac{c^2}{(t+1)^{d}}e^{-\frac{\norm{x-x_n}{2}^2}{3(t+1)}}\dd x \dd t\\
&  \leq \int_0^T\int_{\omega-x_n} c^2 e^{-\frac{\norm{y}{2}^2}{3(T+1)}}\dd y \dd t \\
&  \leq T\int_{(\omega-x_n)\cap (Q_n-x_n)} c^2 e^{-\frac{\norm{y}{2}^2}{3(T+1)}}\dd y \\
& \hspace{2.5cm}+T \int_{(\Omega_L-x_n)\setminus(Q_n-x_n) }  c^2 e^{-\frac{\norm{y}{2}^2}{3(T+1)}}\dd y \\
&  \leq T c^2\abs{\omega\cap Q_n}+ T \int_{(\Omega_L-x_n)\setminus(Q_n-x_n) } c^2 e^{-\frac{\norm{y}{2}^2}{3(T+1)}}\dd y.
\end{align*}
Since 
$Q_n$ exhausts the whole of $\Omega_L$ for $n\to+\infty$, the second term in the last line tends to zero as $n$ goes to infinity, 
and so does the first term due to the choice of $Q_n$. This leads to the desired contradiction. 

\medskip 

We now turn to the case of Neumann boundary conditions. We treat this case with a strategy similar to the one already used. 
As before, we consider the initial value $g_n(x)= K_{\Omega_L}(1,x,x_n)$ so that $g_n(t,x)=e^{t\Delta} g_n(x)=K_{\Omega_L}(1+t,x,x_n)$ 
solves the adjoint system \eqref{eq:heat_equation-adjoint}. 
In order to obtain a contradiction argument we use the following upper and lower Gaussian bounds 
\begin{equation}\label{eq:gaussian-bound-neumann} 
\frac{C_2}{c(d) t^{d/2}}e^{-c_2\frac{\norm{x-y}{2}^2}{t}} \leq K_{\Omega_L}(t,x,y) \leq \frac{C_1}{c(d) t^{d/2}}e^{-c_1\frac{\norm{x-y}{2}^2}{t}}
\end{equation}
valid for some positive constants $c_1, c_2, C_1, C_2$, all $t>0$, all $x,y\in\Omega_L$ (see for example \cite{grygorian:92, li-yau:86}). 
Here $c(d)$ stands for the volume of the Euclidean unit ball centred at zero.

Using the lower bound in \eqref{eq:gaussian-bound-neumann}, the estimate $\abs{x_j -\pi L}^2\leq \pi^2 L^2$ for the first $d-1$ coordinates, 
and the change of variable $y=x_d-x_{n,d}$, we obtain
\begin{align*} 
\int_{\Omega_L}\abs{g_n(T,x)}^2 \dd x & =\int_{\Omega_L} \abs{K_{\Omega_L}(1+T, x, x_n)}^2 \dd x\\
& \geq \frac{C_2^2}{c(d)^2(1+T)^d}\int_{\Omega_L} e^{-2 c_2\frac{\norm{x-x_n}{2}^2}{(1+T)}}\dd x\\
& = \frac{C_2^2}{c(d)^2(1+T)^d}\int_{\TT^{d-1}_L}\int_{\RR} \prod_{j=1}^d e^{-\frac{2c_2 \abs{x_j-x_{n,j}}^2}{(1+T)}} \dd x_d \dd\;(x_1\ldots x_{d-1})\\
& \geq \frac{C_2^2(2\pi L)^{d-1}}{c(d)^2(1+T)^d}\exp\left(\frac{-2c_2(d-1)\pi^2 L^2}{(1+T)}\right)\int_{\RR} e^{\frac{-2c_2\abs{y}^2}{(1+T)}}\dd y,
\end{align*}
i.e. the left hand side of \eqref{eq:observability-def} is bounded away from zero by a constant independent of $n\in\NN$. 

For the right hand side of \eqref{eq:observability-def}, using similar steps as for the Dirichlet case, the upper bound in \eqref{eq:gaussian-bound-neumann}, 
and the estimate $c(d)(1+t)^{1/2}\geq c(d)$ for $t\geq 0$ we have 
\begin{align*} 
\int_0^T \int_\omega \abs{g_n(t,x)}^2 \dd x\dd t & = \int_{0}^T\int_{\omega} \abs{K_{\Omega_L}(1+t, x, x_n)}^2 \dd x \dd t\\ 
& \leq \int_0^T \int_\omega \frac{C_1^2}{c(d)^2 t^{d}}e^{-2c_1\frac{\norm{x-y}{2}^2}{t}} \dd x \dd t \\
& \leq \frac{T C_1^2}{c(d)^2}\left(\abs{\omega\cap Q_n} +\int_{(\Omega_L-x_n)\setminus(Q_n- x_n)} e^{-\frac{2c_1\norm{y}{2}^2}{(1+T)}} \dd x\dd t\right),
\end{align*}
which goes to zero as $n$ goes to infinity and leads to contradiction.


\bigskip

To conclude this section, we show that condition \eqref{eq:necessary-miller} is equivalent to thickness with respect to $\Omega_L$.

\begin{lemma}\label{lemma:equivalence}
Let $\omega\subset\Omega_L$ be a measurable set with $\abs{\omega}>0$. Then,
$\omega$ is not thick with respect to $\Omega_L$ if and only if
there exist a sequence of points $(y_n)_{n\in\NN}$ in $\Omega_L$, a time $T>0$ and a constant $\kappa>1$ such that 
\begin{equation}\label{eq:necessary-miller-1}
-2T\log\left(\int_\omega \exp\left( -\frac{\norm{x-y_n}{2}^2}{2T}\right) \dd x \right) 
      -\kappa \frac{\pi^2 d^2}{4}\left( \frac{T}{d_b(y_n,\partial\Omega_L)} \right)^2 \underset{n\to\infty}{\longrightarrow} + \infty,
\end{equation}
where $d_b(y_n,\partial\Omega_L)=\min \left( \dist(y_n, \partial\Omega_L), \frac{T\pi^2d}{4}\right)$.
\end{lemma}

\begin{proof}

We first assume that $\omega$ is not thick and consider the sequence of hyperrectangles $Q_n$ chosen as in \eqref{eq:cubes} with centre $x_n\in\Omega_L$. 
We show that the sequence $(x_n)_{n\in\NN}$ satisfies condition \eqref{eq:necessary-miller-1} for all $T>0$ and all $\kappa>1$. 
By monotonicity of the exponential, the change of variable $y= x-x_n$ and the fact that $e^{-x}\leq 1$ for all $x>0$, we have 
\begin{align*} 
\int_\omega \exp\left( -\frac{\norm{x-x_n}{2}^2}{2T}\right) \dd x & \leq \int_{(\omega\cap Q_n)- x_n} \exp\left(-\frac{\norm{y}{2}^2}{2T}\right) \dd y \\
& \hspace{2cm}+ \int_{(\Omega_L-x_n)\setminus(Q_n- x_n)} \exp\left(-\frac{\norm{y}{2}^2}{2T}\right) \dd y \\
& \leq \abs{\omega\cap Q_n}+ \int_{(\Omega_L-x_n)\setminus(Q_n- x_n)} \exp\left(-\frac{\norm{y}{2}^2}{2T}\right) \dd y,
\end{align*}
which tends to zero as $n\to\infty$. Therefore
\[
-2T\log \left( \int_\omega \exp\left( -\frac{\norm{x-y_n}{2}^2}{2T}\right) \dd x \right)\underset{n\to\infty}{\longrightarrow} +\infty.
\]
Since $\dist(x_n,\partial\Omega_L)= \pi L$ for all $n\in\NN$, the second summand in \eqref{eq:necessary-miller-1} is only a constant. 
Hence, $(x_n)_{n\in\NN}$ satisfies \eqref{eq:necessary-miller-1}.  

To prove the converse implication we assume that $\omega$ is thick with respect to $\Omega_L$ with parameters $\gamma>0$ and $a\in(\RR_+)^d$ 
and show that \eqref{eq:necessary-miller-1} does not hold. 
Let therefore $(y_n)_{n\in\NN}$ be any sequence of points in $\Omega_L$, $T>0$, $\kappa>1$, and $(P_n)_{n\in\NN}$ be a sequence of hyperrectangles 
with sides of length $a_1,\ldots, a_d$ and such that $y_n\in P_n$. 
Then, $\abs{\omega\cap P_n}\geq \gamma\abs{P_n}>0$ for all $n\in\NN$. 
Let $D=D(a,d)$ the length of the main diagonal of $P_n$. Using the monotonicity of the exponential we have
\begin{align*}
\int_\omega \exp\left( -\frac{\norm{x-y_n}{2}^2}{2T}\right)&  \dd x  \geq \int_{\omega\cap P_n} \exp\left( -\frac{\norm{x-y_n}{2}^2}{2T}\right) \dd x \\
& \geq \exp\left(-\frac{D^2}{2T}\right)\abs{\omega\cap P_n} \geq \exp\left(-\frac{D^2}{2T}\right)\gamma\prod_{j=1}^d a_j >0.
\end{align*}
Consequently, the first summand in \eqref{eq:necessary-miller-1} is bounded from above by a constant for all $n\in\NN$. 
Since the second summand is non-positive for all $n\in\NN$, the sum cannot diverge to $+\infty$. 
\end{proof}

\bibliography{heat-eq-bib} 
\bibliographystyle{plain}

\end{document}